\documentclass[a4paper,11pt]{amsart}
\usepackage[utf8]{inputenc}
\usepackage[english]{babel}
\usepackage{fancyhdr}
\usepackage{amsmath}
\usepackage{amsfonts,amstext}
\usepackage{amssymb,amsthm,amscd,amsxtra,wasysym,graphicx}
\usepackage{mathrsfs,esint,comment}
\usepackage{tikz-cd}
\usepackage{enumitem,mathtools,dsfont}
 \usepackage{palatino,mathpazo}
 \usepackage{palatino}
\usepackage{charter}
\usepackage{xcolor}

\def\Bibtex{{\rm B\kern-.05em{\sc i\kern-.025em b}\kern-0.08em T\kern-.1667em\lower.7ex\hbox{E}\kern-.125emX}}
 \usepackage[breaklinks=true,bookmarks=false,pagebackref]{hyperref}
 \usepackage{hyperref}
\hypersetup{
     unicode=false,          
    pdftoolbar=true,       
    pdfmenubar=true,       
    pdffitwindow=false,     
    pdfstartview={FitH}, 
    pdftitle={Modulus of continuity},   
    pdfauthor={Quang-Tuan Dang},   
    colorlinks=true,  
   linkcolor=purple,         
    citecolor=blue,        
    filecolor=green,      
    urlcolor=blue}

 \usepackage{fullpage}

\usepackage{colonequals}

\usepackage{enumitem}

\numberwithin{equation}{section}

\newtheorem{theorem}{Theorem}[section]
\newtheorem{proposition}[theorem]{Proposition}
\newtheorem{corollary}[theorem]{Corollary}
\newtheorem{lemma}[theorem]{Lemma}
\theoremstyle{definition}
\newtheorem{definition}[theorem]{Definition}
\newtheorem{remark}[theorem]{Remark}

\newtheorem{example}[theorem]{Example}

\newcommand{\Vol}{{\rm Vol}}

 \newcommand{\dist}{\mathop{\mathrm{dist}}\nolimits}

\newcommand{\ddc}{dd^c}

\newcommand{\PSH}{{\rm PSH}}

\newcommand{\capacity}{\mathop{\mathrm{Cap}}\nolimits}

\newcommand{\R}{\mathbb{R}}

\begin{document}
		\title[Modulus of continuity]{Modulus of continuity of Monge--Amp\`ere potentials in big cohomology classes}
	\author{Quang-Tuan Dang, Hoang-Son Do, Hoang Hiep Pham}

			\address{Yau Mathematical Sciences Center, Tsinghua University, Beijing 100084, China}
		\email{dangqt@mail.tsinghua.edu.cn $\&$ dangquangtuan10@gmail.com}
  \address{Vietnam Academy of Science and Technology, Institute of Mathematics, 18 Hoang Quoc Viet road, Cau Giay, Hanoi, Vietnam}
  \email{dhson@math.ac.vn}
  \address{Institute for Artificial Intelligence, University of Engineering and
Technology, Vietnam National University, Hanoi, Vietnam}
 \email{phhiepvn@gmail.com}
	\date{\today}

\begin{abstract} In this paper, we prove a uniform estimate for the modulus of continuity of solutions to the degenerate complex Monge--Amp\`ere equation in big cohomology classes. This improves the previous work of Di Nezza--Lu and the first author.
\end{abstract}

\maketitle

\section{Introduction}

The study of complex Monge-Amp\`ere equations on compact K\"ahler manifolds has attracted considerable interest since Yau’s resolution of the Calabi conjecture. In connection with the Minimal Model Program, the search for singular K\"ahler--Einstein metrics
leads to the study of degenerate complex Monge–Amp\`ere equations; see~\cite{eyssidieux2009singular,berman2019kahler,guedj2017degenerate} and references therein.

Guedj and Zeriahi~\cite{guedj2007weighted} developed the first step of the study of the non-pluripolar Monge--Amp\`ere measure and solved degenerate complex Monge--Amp\`ere equations with rather general measures on the right-hand side. Their approach was later extended to the setting of big cohomology classes by Boucksom, Eyssidieux, Guedj, and Zeriahi~\cite{boucksom2010monge}. Since then, when the right-hand side belongs to $L^p$ for some $p>1$, the H\"older continuity of solutions to degenerate complex Monge--Amp\`ere equations has been intensively studied by many authors ~\cite{kolodziej1998complex,kolodziej2008holder,
hiep2010holder,
eyssidieux2011viscosity,boucksom2010monge,demailly2014holder, dinh2014characterization, 
dang2021continuity,dinh2022complex}. 
The modulus of continuity of the solution plays a crucial role since it is closely related to the geometric properties of the corresponding K\"ahler metrics, such as the uniform bounds for diameter of metrics and Gromov--Hausdorff convergence; see~\cite{Fu-Guo-Song2020-geometric,guo2022-local,guo2021modulus,guo2022-diameter, Guo2024-diameter2,guedj2025-diameter,vu2024continuity,do2023log,Guedj-To-25-green,vu24-diameter,nguyen-vu24-diameter}.

The primary goal of this paper is to study the modulus of continuity of solutions to complex Monge--Ampère equations when the right-hand side is not integrable.
To state our result, we first introduce some notation and terminology. Let $X$ be a compact K\"ahler manifold of dimension $n$ equipped with a K\"ahler form $\omega_X$. We let $d$, $d^c$ denote the real differential operators on $X$ defined by $d:=\partial+\Bar{\partial}$, $d^c:=\frac{i}{2}(\Bar{\partial}-\partial)$ so that $\ddc=i\partial\Bar{\partial}$.
Fix a closed real smooth (1,1)-form $\theta$ representing a big cohomology class. Let $\PSH(X,\theta)$ denote the set of all $\theta$-psh functions. 
We say that the cohomology class $\{\theta\}$ is {\em big} if the set $\PSH(X,\theta-\varepsilon\omega_X)\neq \varnothing$ for some $\varepsilon>0$.
Let $\textrm{Amp}(\theta)$ denote the {\em ample locus} of $\theta$ which is roughly speaking the largest Zariski open subset where $\{\theta\}$ locally behaves like
a K\"ahler class. 

We are interested in studying the regularity of solutions to the complex Monge--Amp\`ere equation type $u\in\PSH(X,\theta)$ satisfying \begin{equation}\label{eq: dcmae}
   (\theta+\ddc u)^n=\nu,\quad\sup_X u=0, 
\end{equation}
where
$\nu$ is a positive measure on $X$ that puts no mass on pluripolar subsets and satisfies the compatibility condition $\nu(X)=\Vol(\theta)$, $u$ is the unknown $\theta$-function, and the left-hand side of~\eqref{eq: dcmae} denotes the non-pluripolar Monge--Amp\`ere product~\cite{bedford1987fine,boucksom2010monge}. 

Let $\mathcal{C}$ denote the set of $\omega_X$-psh functions $u$ normalized by $\int_X u\omega_X^n=0$. This is a
convex compact set in the $L^1(\omega_X^n)$ topology. Define the following distance on $\mathcal{C}$ 
\[\dist(u,v)\coloneqq \|u-v\|_{L^1}, \]
where the $L^1$-norm is with respect to the measure $\omega_X^n$. 
We say that a measure $\mu$ is {\em H\"older continuous}
	 with H\"older constant $B$ and H\"older exponent $0<\beta\leq 1$
	 with respect to $\dist_{L^1}$ on $\mathcal{C}$ if for any $u,v\in\mathcal{C}$
     \[\int_X|u-v|\omega_X^n\leq B\|u-v\|_{L^1}^\beta. \]
Our main result is the following:
\begin{theorem}\label{main} Let $(X,\omega_X)$ be a compact K\"ahler manifold of dimension $n$ and $\theta$ a smooth closed (1,1) form whose cohomology class is big.
	 Let $\mu$ be a H\"older continuous measure 
	 with H\"older constant $B$ and H\"older exponent $0<\beta\leq 1$
	 with respect to $\dist_{L^1}$ on $\mathcal{C}$.
	  Assume that $\psi$ is a quasi-psh function
	 on $X$ satisfying $\int_Xe^{-\psi}d\mu=\Vol(\theta)$ and $\omega_X+a_0dd^c\psi\geq 0$,
	 with $a_0>0$. Suppose that  
	 $u\in\mathcal{E} (X, \theta)$ is a solution to the following equation \begin{equation} (\theta+\ddc u)^n=e^{-\psi}\mu,\quad\sup_X u=0.\end{equation}
	  Then for each $U\Subset{\rm Amp}(\theta)\backslash\{\psi=-\infty\}$, there exists a continuous function $F_U:\R_+\rightarrow\R_+$
	  with $F(0)=0$ such that
	  $$|u(z_1)-u(z_2)|\leq F_U(\dist (z_1, z_2)),$$
	  for every $z_1, z_2\in U$. Moreover, the choice of $F_U$ depends only on  $X$, $U$, $\omega_X$, $n$, $\theta$, $a_0$, $B$, $\beta$, $\sup_U(-\psi)$ and
	  an upper bound function for  $H(a)=\int_Xe^{{2(V_{\theta}-u)}/{a}}d\mu$.
\end{theorem}
Here we notice that $e^{{2(V_{\theta}-u)}/{a}} \in L^1(\mu)$ for every $a>0$ by Lemma~\ref{lem integ phi-u}. In the case $\theta=\omega_X$,
we have the following corollary:
\begin{corollary}\label{coro: main}Let $(X,\omega_X)$ be a compact K\"ahler manifold of dimension $n$.
	Let $\mu$ be a H\"older continuous measure 
	with H\"older constant $B$ and H\"older exponent $0<\beta\leq 1$
	with respect to $dist_{L^1}$ on $\mathcal{C}$.
	Assume that $\psi$ is a quasi-psh function
	on $X$ satisfying $\int_Xe^{-\psi}d\mu=\Vol(\theta)$, $\omega_X+a_0dd^c\psi\geq 0$ and \begin{equation}\label{eq: H}
	    \int_Xh(-\psi)e^{-\psi}\mu\leq C_0,
	\end{equation}
	where $a_0, C_0>0$ are constants and $h: (0, \infty)\rightarrow (0, \infty)$ is an increasing, 
	concave function with $h(\infty)=\infty$. Suppose that  
	$u\in\mathcal{E} (X, \omega_X)$ is a solution to the following equation \begin{equation}\label{eq: cmae_kahler} (\omega_X+\ddc u)^n=e^{-\psi}\mu,\quad\sup_X u=0.\end{equation}
	Then for each $U\Subset X\backslash\{\psi=-\infty\}$, there exists a continuous function $F_U:\R_+\rightarrow\R_+$
	with $F(0)=0$ such that
	$$|u(z_1)-u(z_2)|\leq F_U(\dist (z_1, z_2)),$$
	for every $z_1, z_2\in U$.  Moreover, the choice of $F_U$ depends only on $X$, $U$, $\omega_X$, $n$, $\theta$, $a_0$, $B$, $\beta$, $\sup_U(-\psi)$, $C_0$,
	 and $h$.
\end{corollary}
Let us emphasize that the continuity of the solution $u$ to equation~\eqref{eq: cmae_kahler} was established by Di Nezza and Lu~\cite[Theorem 3.1]{di2017complex}. The contribution of the corollary above is to establish the equicontinuity of the family of solutions corresponding to pairs $(\mu,\psi)$ which satisfy the condition~\eqref{eq: H}, given $h$ and $C_0$. We also note that for each pair $(\mu,\psi)$ with $e^{-\psi}\in L^1(X,\mu)$, there always exists a pair $(h,C_0)$ such that the condition~\eqref{eq: H} holds; cf. Remark~\ref{rmk: hC}. 

\subsection*{Acknowledgment.}  
This work was done while the authors were visiting the Vietnam Institute for Advanced Study in Mathematics (VIASM), and they would like to thank VIASM for its hospitality and support.
\section{Preliminaries} 
Throughout the paper, we let $X$ denote a compact K\"ahler manifold of dimension $n$, equipped with a K\"ahler form $\omega_X$. 
We denote by \(dV_{\omega_X}:={\omega_X^n} \)
the volume form associated with $\omega_X$. For any measure $\mu$ on $X$, we write $L^1(\mu)$ for $L^1(X,d\mu)$.

\subsection{Quasi-plurisubharmonic functions}

Recall that an upper semi-continuous function $ \varphi:X \rightarrow\mathbb{R}\cup\{-\infty\} $
is called {\it quasi-plurisubharmonic} ({\it quasi-psh} for short) if it is locally the sum of a smooth and a plurisubharmonic (psh for short) function. 

We say that $\varphi$ is {\it $\theta$-plurisubharmonic}  ({\it $\theta$-psh} for short) if it is quasi-psh, and $$\theta_\varphi:=\theta+\ddc \varphi\geq 0$$ in the sense of currents, where ${\rm d}=
\partial+\Bar{\partial}$ and ${\rm d}^c=\frac{i}{2\pi}(\Bar{\partial}-\partial)$ so that $\ddc=\frac{i}{\pi}\partial\Bar{\partial}$. 
We let $\PSH(X,\theta)$ denote the set of all $\theta$-psh functions, which are not identically $-\infty$. This set is endowed with 
the weak topology, which coincides  with 
 the $L^1$-topology. By Hartogs' lemma, $\varphi\mapsto\sup_X\varphi$ is continuous in this weak topology, 
it follows that the set of $\varphi\in\PSH(X,\theta)$, with $\sup_X\varphi=0$ is compact. 
We refer the reader to~\cite{demaillycomplex,guedj2017degenerate} for more properties of $\theta$-psh functions.

The cohomology class $\{\theta\}$ is said to be {\em big} if the set $\PSH(X,\theta-\varepsilon\omega_X)$ is not empty for some $\varepsilon>0$.	
	We now assume that $\{\theta\}$ is big unless otherwise specified.
By Demailly's regularization theorem \cite{demailly1992regularization}, we can find a closed positive $(1,1)$-current $T_0\in \{\theta\}$ such that $$T_0=\theta+\ddc\Psi_0\geq \varepsilon_0\omega_X,$$ for some $\varepsilon_0>0$, where $\Psi_0$ is a quasi-psh function with \emph{analytic singularities}, i.e., locally 
	$$	\Psi_0=c\log\left[\sum_{j=1}^{N}|f_j|^2\right]+g,
	$$
	where the $f_j$'s are holomorphic functions and $g$ is bounded. Such a current $T_0$ is then smooth on a Zariski open subset $X\setminus\{\Psi_0=-\infty\}$. We thus define the {\em ample locus} $\textrm{Amp}(\theta)$ of $\theta$ as the largest such Zariski open subset (which exists by the Noetherian property of closed analytic subsets; cf.~\cite{boucksom2004divisorial}).

Given $\varphi,\psi\in \PSH(X,\theta)$, we say that $\varphi$ is {\it less singular} than  $\psi$, and denote by $\psi\preceq\varphi$, if  there exists a constant $C$ such that $\psi\leq \varphi+C$ on $X$. We say that $\varphi,\psi$ have the {\em same singularity type}, and denote by $\varphi\simeq\psi$ if $\varphi\preceq\psi$ and $\psi\preceq \varphi$. 
There is a
natural least singular potential in $\PSH(X,\theta)$ given by
\begin{align*}
	V_{\theta}:=\sup\{\varphi\in \PSH(X,\theta): \varphi\leq 0\}.
	\end{align*}
A function $\varphi$ is said to have {\em minimal singularities} if it has the same singularity type as $V_\theta$. In particular, $V_\theta=0$ if $\theta$ is semi-positive.
We see that $V_\theta$ is locally bounded in the ample locus.

\subsection{Demailly's regularization}

   We consider the exponential mapping $ \exp_x: T_x X\ni \zeta\to X$ as follows: $\exp_x(\zeta)=\gamma(1)$ where $\gamma:[0,1]\to X$ is the geodesic starting from $x=\gamma(0)$ with the initial velocity  $\gamma'(0)=\zeta$. In the Euclidean space $\mathbb{C}^n$, the exponential map $\exp_x(\zeta)=x+\zeta$. 

 Following~\cite{demailly1982estimations,demailly1994regularization}, for a quasi-psh function $u$,
we define its {\em $\delta$-regularization} $\rho_\delta u=\Psi(u)(z,\delta)$ where
\begin{equation}\label{phie}
\Psi(u)(z,w)=\int_{\zeta\in T_{z}X}
u\big(\exp_z(w\zeta)\big)\chi\big({|\zeta|^2_{\omega_X }}\big)\,dV_{\omega_X}(\zeta),\ \delta>0,
\end{equation}
where $\chi: \mathbb R_{+}\rightarrow\mathbb R_{+}$ is defined by
\begin{center}
$\chi(t)=\begin{cases}\frac {\eta}{(1-t)^2}\exp(\frac 1{t-1})&\ {\rm if}\ 0\leq t\leq 1,\\0&\
	{\rm if}\ t>1,\end{cases}$
\end{center}
with a suitable constant $\eta$ such that
$\int_{\mathbb C^n}\chi(\Vert z\Vert^2)\,dV(z)=1$, where $dV$ denotes the Lebesgue measure on $\mathbb{C}^n$. The $\delta$-regularization $\rho_\delta u$ can be written by
\[\rho_\delta u(z)=\frac{1}{\delta^{2n}}\int_{\zeta\in T_{z}X}
u\big(\exp_z(\zeta)\big)\chi\bigg(\frac{|\zeta|^2_{\omega_X }}{\delta^2}\bigg)\,dV_{\omega_X}(\zeta).\]
Intuitively, this corresponds to the familiar convolution with smoothing kernels. Actually, in the case of $\mathbb{C}^n$ endowed with the Euclidean metric, this is exactly the smoothing convolution; see~
\cite[Remark 4.6]{demailly1994regularization}.

The following lemma is a combination of \cite[Theorem 4.1]{demailly1994regularization} and \cite[Lemma 1.12]{berman2012regularity}.
\begin{lemma}\label{lem kiselman}
	Let $u$ be a $\theta$-psh function and define the Kiselman-Legendre transform with level $c$ by
	\begin{equation}\label{kisleg}
	\Phi_{c,\delta}\coloneqq\inf _{0< t\leq \delta }\Big[\rho_tu(z)+K (t^2-\delta^2)+K (t-\delta)
	-c\log\Big(\frac t{\delta}\Big)\Big].
	\end{equation}	
	Then for  some positive constants $0<\delta_0<1$ and $K>1$  depending on the curvature tensor of $\omega_X$ on $X$,
	 the function $\rho_tu(z)+Kt^2$ is increasing in $t\in (0, \delta_0]$ and
	one has the following estimate for the complex Hessian:
	\begin{equation}\label{hessest}
	\theta+dd^c \Phi_{c,\delta}\geq -(Ac+K^2\delta)\,\omega_X,
	\end{equation}
	where $A$ is a lower bound of the negative part of the bisectional curvature of $\omega_X$.
\end{lemma}
  \begin{proof} 
 The proof is identical to that of~\cite[Lemma 4.1]{kolodziej2019stability}, which is still valid without the boundedness of $u$.
 \end{proof}

\begin{remark}
   The above lemma is essential for our study of the modulus of continuity of solutions to complex Monge--Ampère equations. We follow the same strategy as in the proof of~\cite[Theorem D]{demailly2014holder}. 
    However, the function $\Phi_{c,\delta}$ differs from the one used there by the addition of the term $K(t - \delta)$, as introduced in~\cite[Remark 4.7]{demailly1994regularization}. This extra term is needed to handle the mixed term $|dz||dw|$ that arises in the estimates; cf.~\cite{kolodziej2019stability}.
\end{remark}
\begin{lemma}\label{Jen}
 Let $u\in\PSH(X,\theta)$.
 If
	$\rho_{\delta} u$ is the regularization of
	$u$ defined as in (\ref{phie}) then for $\delta$ small enough we have
	$$
	\int_{X} \frac{\rho_{\delta} u-u}{\delta ^2 }\omega_X^n \leq C\|u\|_{L^1(X,dV_X)} ,
	$$
	where $C$ only depends on $n$, $X$, $\omega_X$.
\end{lemma}
\begin{proof}
    The proof follows verbatim from that of~\cite[Lemma 2.3]{demailly2014holder}.
\end{proof}
\subsection{Non-pluripolar product}

Let $\theta^1,\ldots,\theta^n$ be closed smooth real (1,1) form representing big cohomology classes, and $\varphi_j\in\PSH(X,\theta^j)$. Following the construction of Bedford--Taylor~\cite{bedford1987fine}, it has been shown in~\cite{boucksom2010monge} that for each $k\in\mathbb{N}$,
\[ \mathbf{1}_{\cap_j\{\varphi_j>V_{\theta^j}-k\}}(\theta^1+\ddc{\max(\varphi_1,V_{\theta^1}-k)})\wedge\cdots\wedge (\theta^n+\ddc{\max(\varphi_n, V_{\theta^n}-k)})\] is well-defined as a  Borel positive measure with finite total mass. The sequence of these measures is non-decreasing in $k$ and it converges weakly to the so-called {\em Monge--Amp\`ere product}, denoted by \[ (\theta^1+\ddc{\varphi_1})\wedge \cdots\wedge (\theta^n+\ddc{\varphi_n}),\] 
which does not charge pluripolar sets by definition. In particular, $\theta^1=\cdots=\theta^n=\theta$ and $\varphi_1=\cdots=\varphi_n$ we obtain the non-pluripolar
Monge--Amp\`ere measure of $\varphi$, denoted by $(\theta+\ddc\varphi)^n$ or simply by $\theta_\varphi^n$.

The {\em volume} of a big cohomology class $\{\theta\}$ is given by the total mass of the non-pluripolar Monge--Ampère
measure of $V_\theta$, i.e.,  $${\rm \Vol}(\theta):=\int_{X}\theta_{V_\theta}^n.$$
  We say that $\varphi\in\PSH(X,\theta)$ has {\it full Monge--Amp\`ere mass} if $\int_X\theta_{\varphi}^n=\Vol(\theta)$. We let \begin{align*}
      \mathcal{E}(X,\theta):=\left\{\varphi\in\PSH(X,\theta):\int_X\theta_{\varphi}^n=\Vol(\theta) \right\}
  \end{align*}
  denote the set of $\theta$-psh functions with full Monge--Amp\`ere mass.
  Note that $\theta$-psh functions with minimal singularities have full  Monge--Amp\`ere mass (see \cite[Theorem 1.16]{boucksom2010monge} for more details), but the converse is not true. 

\smallskip 

Given a measurable function $f:X\to\mathbb{R}$, we define the {\em $\theta$-psh envelope} of $h$ by
\begin{equation*}
P_\theta(f):=(\sup\{u\in\PSH(X,\theta): u\leq f\;\text{on}\, X \})^*,
\end{equation*} where the star means we take the upper semi-continuous regularization. 

Given a $\theta$-psh function $\phi$, 
Ross and Witt-Nystr\"om~\cite{ross2014analytic} introduced the ``rooftop envelope'' as follows
\[P_\theta[\phi](f)= \left( \lim_{C\to +\infty}P_\theta(\min(\phi+C,f))\right)^*. \]
When $f=0$ we simply write $P_\theta[\phi]$. 
A function $\phi\in\PSH(X,\theta)$ is called a {\em model potential} if $$\int_X\theta_\phi^n>0\quad\text{and}\quad\phi=P_\theta[\phi].$$

We recall the notion of Monge--Ampère capacity introduced in~\cite{di2017complex,di2015generalized}.
For $\phi \in \PSH(X,\theta)$ and a Borel subset $E \subset X$, the {\em relative Monge–Ampère capacity} is defined by
    \begin{equation*}
        \capacity_{\phi}(E):=\sup\left\{\int_E\theta_u^n: u\in\PSH(X,\theta),\;\phi-1\leq u\leq \phi \right\}.
    \end{equation*}
This definition recovers the Monge--Ampère capacity used in~\cite{boucksom2010monge} when $\phi = V_\theta$. In that case, we denote it simply by $\capacity_\theta$.

\begin{definition}
A family of positive measures $\{\mu_i \}_{i\in I}$ on $X$ is said to be uniformly absolutely
continuous with respect to $\phi$-capacity if, for every $\varepsilon>0$, there exists $\delta>0$ such that for each Borel subset $E\subset X$ satisfying $\capacity_\phi(E)<\delta$ the inequality $\mu_i(E)<\varepsilon$ holds for all $i\in I$. We denote this by $\mu_i<<\capacity_\phi$. In particular, all such measures must
vanish on pluripolar sets.
\end{definition}

\subsection{H\"older continuous measures}
\begin{definition}\label{def: holder}
	A positive measure $\mu$ on $X$ is said to be $\PSH(X, \theta)$-H\"older continuous
	(or, for short, H\"older continuous) if there exist $B>0$ and $0<\beta\leq 1$ such that
	\begin{equation*}
	    \int_{X}(u-v)d\mu\leq B \left(\int_X|u-v|\omega_X^n\right)^{\beta},
	\end{equation*}
for every $u, v\in\mathcal{C}:=\{w\in \PSH(X, \theta): \int_X w\omega_X^n=0 \}.$ In this case, we say that $B$ is
the H\"older constant and
$\beta$ is the H\"older exponent of $\mu$ with respect to ${\rm dist}_{L^1}$ on $\mathcal{C}$. We let $\mathcal{M}(B,\beta)$ denote the set of such measures.
\end{definition}
\begin{lemma}[{\cite[Lemma 3.3]{dinh2014characterization}}]\label{lem DN14 def holder}
		Assume $\mu$ is a H\"older continuous measure 
		with H\"older constant B and H\"older exponent $0<\beta\leq 1$
		 with respect to $\dist_{L^1}$ on $\mathcal{C}$. Then there exists $C>0$ depending only on $B$ such that
		\begin{equation*}
		    \|u-v\|_{L^1(\mu)}\leq C\max\big\{\|u-v\|_{L^1(\omega_X^n)}, \|u-v\|_{L^1(\omega_X^n)}^{\beta} \big\},
		\end{equation*}
	for all $u, v\in \PSH(X, \theta)$. In particular, if a family $\mathcal{F}$  of $\theta$-psh functions
	is a relatively compact subset of $L^1(\omega_X^n)$ then there exists $C_{\mathcal{F}}>0$ such that
	\begin{equation*}
	    \|u-v\|_{L^1(\mu)}\leq C_{\mathcal{F}} \|u-v\|_{L^1(\omega_X^n)}^{\beta},
	\end{equation*}
	for all $u, v\in\mathcal{F}$.
\end{lemma}

\begin{proposition}\label{prop exp integrable}
	Let $\mu$ be a H\"older continuous measure on $X$
	with H\"oler constant $B$ and H\"older exponent $0<\beta\leq 1$ with respect to ${\rm dist}_{L^1}$ on $\mathcal{C}$. Assume that
	$u, v$ are $\theta$-psh functions satisfying
	\begin{equation*}
	    \int_Xe^{m(u- v)}\omega_X^n<C_0,
	\end{equation*}
for positive constants $m$, $C_0$. Then, for every $0<\gamma<\beta$, there exists $C>0$ depending only on $C_0, 
m, B,\beta$ and $\gamma$ such that
\begin{equation*}
    \int_Xe^{\gamma m(u-v)}d\mu<C.
\end{equation*}
\end{proposition}
\begin{proof} We adapt the proof of~\cite[Proposition 4.4]{dinh2014characterization}.
	Denote $w=v-u$. We have
	\begin{equation}\label{eq0 exp integ}
		\omega_X^n(\{w<-M\})\leq\int_Xe^{-m(w+M)}\omega_X^n\leq C_0e^{-mM},
	\end{equation}
for every $M>0$. Denote 
	$$w_M=\max\{w, -M \}=\max\{v, u-M\}-u.$$
We have $w_M-w=\max\{v, u-M\}-v$ is the difference of two $\theta$-psh functions. 
By Lemma \ref{lem DN14 def holder}, there exists $C_1>0$ depending only on $B$ such that
\begin{equation}\label{eq1 exp integ}
\int_X(w_M-w)d\mu\leq C_1\max\left\{\|w_M-w\|_{L^1(\omega_X^n)}, \|w_M-w\|_{L^1(\omega_X^n)}^{\beta} \right\}.
\end{equation}
Moreover, we also have
\begin{equation}\label{eq2 exp integ}
\mu(\{w<-M-1\})\leq\int_X(w_M-w)d\mu,
\end{equation}
and by \eqref{eq0 exp integ}, 
\begin{equation}\label{eq3 exp integ}
\begin{split}
    \int_X(w_M-w)\omega_X^n&=\int_0^{\infty}\omega_X^n(\{w<-M-t\})dt\\
    &\leq C_0\int_0^{\infty}e^{-m(M+t)}dt
=\dfrac{C_0e^{-mM}}{m}.
\end{split}
\end{equation}
Combining \eqref{eq1 exp integ}, \eqref{eq2 exp integ} and \eqref{eq3 exp integ}, we get

$$	\mu(\{w<-M-1\})\leq C_1\max\left\{\dfrac{C_0e^{-mM}}{m}, \left(\dfrac{C_0e^{-mM}}{m}\right)^{\beta}\right\}
	\leq C_2\cdot e^{-\beta mM},$$
where $C_2=C_1\max\{\frac{C_0}{m}, (\frac{C_0}{m})^{\beta}\}$. Then, for every $0<\gamma<\beta$, we obtain
\begin{align*}
	\int_Xe^{-\gamma m w}d\mu-\mu (X)
	&=\gamma m\int_0^{\infty}\mu(w<-t)e^{\gamma m t}dt\\
	&\leq 	\gamma m C_2e^{\beta m}\int_0^{\infty}e^{(-\beta+\gamma)m t}dt\\
	&=\frac{C_2\gamma e^{\beta m}}{\beta-\gamma}.
\end{align*}
\end{proof}

\begin{lemma}\label{lem integ phi-u} Let $\mu\in\mathcal{M}(B,\beta)$ be a H\"older continuous measure on $X$
	and $u\in\mathcal{E}(X, \theta)$.
	Then
		$$\int_Xe^{m(V_{\theta}-u)}d\mu<+\infty,$$
for every $m>0$.
 \end{lemma} 
 \begin{proof} 
     It follows from~\cite[Proposition 2.10]{dang2021continuity} that for any $b>0$, $P_{\omega_X}(b(u-V_\theta))\in\mathcal{E}(X,\omega_X)$. We remark here that $u-V_\theta$ is well-defined outside a pluripolar set. 
     In particular, for $0<\gamma<\beta$, we have
     \[\int_X e^{m\gamma^{-1}(V_\theta-u)}\omega_X^n\leq \int_X e^{-m\gamma^{-1}P_{\omega_X}(u-V_\theta)}\omega_X^n<+\infty, \]
     as follows from Skoda's integrability theorem; cf.~\cite[Theorem 2.50]{guedj2017degenerate}. We, therefore, apply Proposition~\ref{prop exp integrable} to conclude.
 \end{proof}

\begin{proposition}\label{the exp cap}
	Let $\phi\in \PSH(X, \theta)$ be a model potential with $\int_X\theta_{\phi}^n>\varrho>0$. 
    Let $\mu\in\mathcal{M}(B,\beta)$.
    Then, there exist constants
	$0<\gamma<1$ and $C>0$ depending only on $X, \omega_X, \theta, B$ and $\beta$ such that
	\begin{equation}\label{eq0 the exp cap}
		\mu (E)\leq C\,
		\exp\left(-\gamma\left(\dfrac{\varrho}{\capacity_{\phi}(E)}\right)^{1/n} \right),
	\end{equation}
	for all Borel sets $E\subset X$. 
\end{proposition}
\begin{proof} The case when $\theta=\omega_X$ and $\phi=0$ was shown in \cite[Propositions 2.4, 4.4]{dinh2014characterization}. For the relative case,
   the proof relies on the arguments in~\cite[Proposition 4.30]{darvas2018monotonicity}, so we sketch it here. 
   According to~\cite[Section 4A2]{darvas2018monotonicity}, for any Borel set $E$, we define the global $\phi$-extremal function of $(E,\theta,\phi)$ by
   \[V_{E,\phi}=\sup\{ u\in\PSH(X,\theta,\phi): u\leq \phi \;\text{on}\, E\}. \]
 Set $M_\phi(E)=\sup_X V^*_{E,\phi}$, we may assume $M_\phi(E)\geq 1$.  Thanks to~\cite[Lemma 4.9]{darvas2018monotonicity}, we have the Alexander--Taylor type inequality
 \[ \left(\frac{\varrho}{\capacity_{\phi}(E)}\right)^{1/n} \leq M_\phi(E),\]
 hence,
 \[ \exp(-M_\phi(E))\leq \exp\left(-\left(\frac{\varrho}{\capacity_{\phi}(E)}\right)^{1/n} \right). \] 
  As follows from~\cite[Proposition 4.4]{dinh2014characterization}, $\mu$ is weakly moderate, i.e. there are constants $\alpha=\alpha(\beta)>0$ and $C=C(B,\beta)>0$ such that $\int_X \exp(-\alpha \varphi)d\mu\leq C$ for every $\varphi\in\mathcal 
   C$. We apply this to $V_{E,\phi}^*-\int_X V_{E,\phi}^*\omega_X^n$ to obtain 
   \[\int_X\exp(-\alpha V_{E,\phi}^*)d\mu\leq C\cdot \exp\bigg(-\alpha\int_XV_{E,\phi}^*\omega_X^n\bigg). \]
   Since $V_{E,\phi}^*\leq 0$ on $E$ a.e. and $\int_X V_{E,\phi}^*{\omega_X}^n\geq M_\phi(E)-C_{\omega_X}$ (see, e.g.,~\cite[Proposition 8.5]{guedj2017degenerate}) we have
   \[ \mu(E)\leq C\cdot\exp\left(-\gamma\left(\frac{\varrho}{\capacity_{\phi}(E)}\right)^{1/n} \right),\] for $\gamma>0$. 
\end{proof}

\begin{proposition} \label{prop: unif-capa}
Fix $a_0>0$, $ C_0>0$ and $h: (0, \infty)\rightarrow (0, \infty)$ is an increasing concave function with $h(\infty)=\infty$.
Let $\mathcal{N}=\mathcal{N}(B,\beta, a_0,C_0,h)$ be the set of probability measures $\nu$ on $X$ such that $\nu=e^{-\psi}\mu$ with $\mu\in\mathcal{M}(B,\beta)$, $\psi\in\PSH(X, a_0\omega_X)$ and  \begin{equation}\label{eq: assumption H}
    \int_Xh(-\psi)e^{-\psi}{\rm d}\mu\leq C_0.
\end{equation}  
Then there exists a continuous function $g\colon [0,\infty)\to [0,\infty)$ with $g(0)=0$ such that for all Borel sets $E$,
\[\nu(E)\leq g(\capacity_{\omega_X}(E)) \quad\text{for all}\,\nu\in\mathcal{N}.\]
In particular, the family of measures $(\nu)_{\nu\in\mathcal{N}}$ is uniformly absolutely
continuous with respect to capacity.
\end{proposition}

\begin{proof}
    For any Borel subset $E\subset X$, if $\capacity_{\omega_X}(E)=0$, then $E$ is a pluripolar set, and consequently $\nu(E)=0$. In this case, the desired inequality holds trivially.
We may therefore assume $\capacity_{\omega_X}(E)>0$. We have for $k>0$
    \begin{align*}
        \int_E e^{-\psi}d\mu&=\int_{E\cap\{\psi\geq -k \} } e^{-\psi}d\mu+\int_{E\cap \{\psi<-k \}}e^{-\psi}d\mu\\
        &\leq e^k\mu(E)+ \frac{1}{h(k)}\int_{E\cap \{\psi<-k \}}h(-\psi)e^{-\psi}d\mu\\
        &\leq C(e^k\capacity_{\omega_X}(E)+{h(k)}^{-1}),
    \end{align*}
    where the last inequality follows from Proposition~\ref{the exp cap}. 
    If $\capacity_{\omega_X}(E)=0$ then $E$ is a pluripolar set, so $\nu(E)=0$, the desired inequality is trivial.
 Otherwise, taking $k\coloneqq\log \frac{1}{\sqrt{\capacity_{\omega_X}(E)}}>0$ we get that $e^{-\psi}\mu(E)\leq g(\capacity_{\omega_X}(E))$ where
    \[ g(t)\coloneqq C(t^{1/2}+h(\log (t^{-1/2}))^{-1}).\]
     Otherwise, since $\frac{1}{\sqrt{\capacity_{\omega_X}(E)}}\leq 1$, the choice of $g$ ensures that $\nu(E)\leq C\leq g(\capacity_{\omega_X}(E))$.
\end{proof}
\begin{remark}\label{rmk: hC}
    For each non-pluripolar measure $e^{-\psi}\mu$ with $\int_Xe^{-\psi}d\mu<\infty$, there always exists a concave increasing function $h\colon \mathbb{R}^+
    \to\mathbb R^+$ and $C>0$ such that $\int_X h(-\psi)e^{-\psi}d\mu\leq C$ as follows from~\cite[Lemma 3.3]{boucksom2010monge}. 
\end{remark}
We end this section with some examples of H\"older continuous measures.
\begin{example}
   Dinh and Nguy\^en~\cite{dinh2014characterization} gave a characterization of H"older continuous measures in the sense of Definition~\ref{def: holder}. Specifically, they proved that a measure $\mu$ is H\"older continuous if and only if it can be expressed as the Monge--Amp`ere measure of a H\"older continuous $\omega_X$-psh function, i.e., $\mu = (\omega_X + \ddc\varphi)^n$ for some $\varphi\in\PSH(X,\omega_X)$ that is H\"older continuous. As a consequence, one obtains explicit examples of such measures; several are listed below.
    \begin{itemize}
        \item $\mu=f\nu$ where $\nu$ is a H\"older continuous measure and $f\in L^p(\nu)$ for some $p>1$; cf.~\cite{kolodziej2008holder,demailly2014holder}. 
        \item If there are constants $C>0$ and $\alpha>0$ such that $\mu(B(z,r))\leq Cr^{2n-2+\alpha}$ for $B(z,r)$ denoting the ball centered at $z$ with radius $r>0$, then $\mu$ is H\"older continuous; cf.~\cite{hiep2010holder,demailly2014holder}.
        \item If $\mu$ is a positive Radon measure compactly supported on a immersed $\mathcal{C}^3$ real submanifold $K$ of $X$ of real codimension $d>0$, then $\mu$ is H\"older continuous; cf.~\cite{Vu2018complex}.
    \end{itemize}
\end{example}
\begin{example} We construct an example of a measure with \emph{radial singularities}, inspired by~\cite[Section 1.3]{di2021finite} and~\cite[Section 5]{guedj2025-diameter}.
Let $X=\mathbb C\mathbb P^n$ be a complex projective manifold of dimension $n$, equipped with
the Fubini--Study metric $\omega_X=\omega_{\rm FS}$. We assume that ${\varphi}$ is a $\omega_{\rm FS}$-psh function on $X$ which has a radial singularity at a point $p$, i.e., it is invariant under the group $U(n,\mathbb C)$ in the neighborhood of $p$. Choosing a local chart biholomorphic to the unit ball $B\subset \mathbb{C}^n$, with $p$ corresponding to the origin. Locally, the function $\varphi$ can be written as $\varphi=u-\frac{1}{2}\log[1+\|z\|^2]$ for some psh function $u$.

We therefore consider a psh function $u\coloneqq\chi\circ L$ defined near the origin in $\mathbb C^n$, where $L(z)\coloneqq\log \|z\|$ and $\chi\colon\mathbb R^-\to \mathbb R^- $ is a convex increasing function. A computation shows that
\[ (\ddc u)^n=\frac{c_n (\chi'\circ L)^{n-1}\chi''\circ L}{\|z\|^{2n}}\omega_X^n=\colon e^{-\psi_\chi}\omega_X^n. \]
To simplify later estimates, we assume that $\chi(t)$ does not go to zero too fast at infinity $-\infty$, that is, $\chi'(t),\chi''(t)\geq e^{Ct}$ near $t=-\infty$. Under this hypothesis, one finds $\psi_\chi\sim \log\|z\|$ near the origin. Hence,
the function $h$ satisfies $\int_X h(-\psi_\chi)e^{-\psi_\chi}\omega_X^n<\infty $ if and only if
\[\int_{-\infty}^{-1} h(-t)\chi'(t)^{n-1}\chi''(t)dt<\infty.\]
Below, we provide several families of functions $\chi$ for which the condition holds with $h(s)=s^\varepsilon$, where $0<\varepsilon\ll 1$.
\begin{enumerate}
    \item Consider $\chi_a(t)=(-t)^a$ for $a\in (0,1)$.  We have
    \[\int_{-\infty}^{-1}h(-t)\chi'(t)^{n-1}\chi''(t)dt\lesssim\int_{-\infty}^{-1} h(-t)\cdot (-t)^{-n(1-a)-1}dt<\infty, \]
    for $h(s)=s^\varepsilon$ with $0<\varepsilon<<1$ and $a<\frac{n-\varepsilon}{n}$.
    \item Consider $\chi_a(t)=-(\log (-t))^a$, where $a>0$. We have 
  \[\int_{-\infty}^{-1}h(-t)\chi'(t)^{n-1}\chi''(t)dt\lesssim\int_{-\infty}^{-1} h(-t)\cdot (-t)^{-n-1}(\log(-t))^{na-n}dt<\infty,\]
  for $h(s)=s^\varepsilon$ with $0<\varepsilon<<1$.
  \item $\chi(t)=-\log(\log(-t))$. We have
  \[\int_{-\infty}^{-1}h(-t)\chi'(t)^{n-1}\chi''(t)dt\lesssim\int_{-\infty}^{-1} h(-t)\cdot (-t)^{-n-1}(\log(-t))^{-n}dt<\infty,\]
   for $h(s)=s^\varepsilon$ with $0<\varepsilon<<1$.
\end{enumerate}
\end{example}

\section{Proof of the main theorem}
 We prove some lemmas which will be used to prove the main theorem.


\begin{lemma}\label{lem t^n cap}
	Assume $u, v$ are negative $\theta$-psh functions such that
	\begin{itemize}
		\item $v\leq P_{\theta}[u]$;
		\item $v=\lambda v_1+(1-\lambda) v_2$, where  $v_1, v_2$ are $\theta$-psh functions
		with $v_2$ having model singularity type and $\lambda\in (0, 1)$. 
	\end{itemize}
Then, for every $s>0$ and $0\leq t\leq 1-\lambda$, we have
	$$t^n \capacity_{v_2}\{u<v-s- t\}\leq \int_{\{u<v-s\}}\theta_u^n.$$
\end{lemma}
\begin{proof} The proof follows closely that of \cite[Lemma 4.31]{darvas2018monotonicity}, which itself originates in \cite[Lemma 4.2]{boucksom2010monge}.
	 Let $\varphi$ be a $\theta$-psh function such that $v_2-1\leq\varphi\leq v_2$. For every
	 $s>0$ and $0\leq t\leq 1$, we have
	 \begin{align*}
	 	\{u<v-s-(1-\lambda) t\}\subset \{u< \lambda v_1+(1-t)(1-\lambda) v_2+(1-\lambda) t\varphi-s \}
	 	 \subset \{u<v-s\}.
	 \end{align*}
 Setting $\widehat{v}:=\lambda v_1+(1-t)(1-\lambda) v_2+(1-\lambda) t\varphi$, we have
 \begin{align*}
 	(1-\lambda)^nt^n\int_{\{u<v-s-(1-\lambda) t\}}\theta_{\varphi}^n
 	\leq \int_{\{u<v-s-(1-\lambda) t\}}\theta_{\widehat{v}}^n
 	\leq \int_{\{u<\widehat{v}-s\}}\theta_{\widehat{v}}^n,
 \end{align*}
for all $s>0$ and $0\leq t\leq 1$. Since $\widehat{v}\leq P_\theta[u]$
it follows from the comparison principle
 \cite[Lemma 2.3]{darvas2021log} that 
 \begin{align*}
   \int_{\{u<\widehat{v}-s\}}\theta_{\widehat{v}}^n  \leq \int_{\{u<\widehat{v}-s\}}\theta_{u}^n
 	\leq \int_{\{u<v-s\}}\theta_{u}^n.
 \end{align*}
 Since $\varphi$ was taken arbitrarily, it follows that
 	$$(1-\lambda)^nt^n \capacity_{v_2}\{u<v-s-(1-\lambda) t\}\leq \int_{\{u<v-s\}}\theta_u^n,$$
for all $s>0$ and $0\leq t\leq 1$. Substituting $t\mapsto (1-\lambda) t$, we obtain the desired inequality.
\end{proof}
\begin{lemma}\label{lem tg(t+s)}
	Let $g: \R^+\rightarrow\R^+$ be a non-increasing, right-continuous function such that, 
	\begin{align*}
		tg(t+s)\leq C(g(s))^{1+\alpha},
	\end{align*}
for every $s>0$ and $0\leq t\leq 1-\lambda$, where $C>0$, $\alpha> 0$ and $\lambda\in (0, 1)$ are given.
Assume that there exist $S_0>0$ and $0\leq t_0\leq 1-\lambda$ satisfying $g(S_0)^{\alpha}<\frac{t_0}{2C^{1/n}}$.
Then $g\left(S\right)=0$ for all $S\geq S_0+\frac{t_0}{1-2^{-\alpha}}$.
\end{lemma}
\begin{proof}
     We refer to \cite[pages 614–615]{eyssidieux2009singular} for a proof.
\end{proof}
\begin{lemma}\label{lem DL17}
	Let $u\in \mathcal{E} (X, \theta)$ such that $\theta_u^n=e^{-\psi}d\mu$, where $\mu$ is a H\"older continuous measure with H\"older constant $B$ and H\"older exponent $0<\beta\leq 1$ on $\mathcal{C}$
	and $\psi$ is a negative quasi-psh function. Suppose that $v$ is a $\theta$-psh function such that
	$v\leq V_{\theta}+a\psi$ for some $a>0$. Assume that $\int_X\theta_{V_{\theta}}^n>\varrho>0$, where $\varrho$ is a constant. Then,  there exists $C=C(n, X, \omega_X, \theta, \varrho, B, \beta)>0$ such that,
	for every $0<\lambda\leq  1/2$,
	\begin{align*}
		u\geq  \lambda v+(1-\lambda) V_{\theta}-C a\lambda \left(1+\log_+\int_Xe^{\frac{2(V_\theta-u)}{a\lambda}}d\mu\right)-2.
	\end{align*}
\end{lemma}
Here we note that $e^{\frac{2(V_{\theta}-u)}{a\lambda}} \in L^1(\mu)$ by Lemma~\ref{lem integ phi-u}. Denote by $\log_+(x)=\max(\log x,0)$.
\begin{proof}
	Denote $\widehat{v}\coloneqq \lambda v+(1-\lambda) V_{\theta}$.
	 We apply \ref{lem t^n cap} with $v_1=v$ and $v_2=V_\theta$ to obtain that for all $s>0$ and $0\leq t\leq 1-\lambda$,
	\begin{align*}
		t^n\capacity_{\theta}(u<\widehat{v}-s-t)\leq \int_{\{u<\widehat{v}-s\}}\theta_u^n
		=\int_{\{u<\widehat{v}-s\}}e^{-\psi}d\mu.
	\end{align*}
Since $v\leq V_{\theta}+a\psi$, we have $-\psi\leq \frac{V_\theta-\widehat v}{a\lambda}$.
It follows that
\begin{align*}
t^n\capacity_{\theta}(u<\widehat{v}-s-t)&\leq \int_{\{u<\widehat{v}-s\}}e^{\frac{V_{\theta}-\widehat{v}}{a\lambda}}d\mu\\
&\leq \int_{\{u<\widehat{v}-s\}}e^{\frac{V_{\theta}-u-s}{a\lambda}}d\mu.
\end{align*}
Then, by H\"older's inequality, we obtain
\begin{equation}\label{eq1 proof lem DL17}
t^n\capacity_{\theta}(u<\widehat{v}-s-t)\leq 
\left( \int_{X}e^{\frac{2(V_{\theta}-u-s)}{a\lambda}}d\mu\right)^{{1}/{2}}
\left( \int_{\{u<\widehat{v}-s\}}d\mu\right)^{{1}/{2}}.
\end{equation}
Since the measure $\mu$ is H\"older continuous, it follows from Proposition~\ref{the exp cap} that there exists $C_1>0$
depending on $X, \omega_X, \theta, \varrho, B$ and  $\beta$ such that
\begin{center}
	$\mu(K)\leq C_1^{2n} [\capacity_{\theta}(K)]^4,$
\end{center}
for every Borel set $K\subset X$, using that $\exp(e^{-1/x})=O(x^4)$ for all $x>0$. 
Then, by \eqref{eq1 proof lem DL17}, we have
\begin{equation}\label{eq2 proof lem DL17}
t^n\capacity_{\theta}(u<\widehat{v}-s-t)\leq C_1^n e^{\frac{-s}{a\lambda}}\left( \int_{X}e^{\frac{2(V_{\theta}-u)}{a\lambda}}d\mu\right)^{{1}/{2}}
[\capacity_{\theta}(u<\widehat{v}-s)]^2.
\end{equation}
Set $g(s)\coloneqq[\capacity_{\theta}(u<\widehat{v}-s)]^{1/n}$ and
 $C_2\coloneqq\left(\int_{X}e^{\frac{2(V_{\theta}-u)}{a\lambda}}d\mu\right)^{{1}/{2n}}$.
  By \eqref{eq2 proof lem DL17},
 we have
 \begin{equation}\label{eq3 proof lem DL17}
 	tg(t+s)\leq C_1C_2g(s)^2,
 \end{equation}
for every $s>0$ an $0\leq t\leq 1-\lambda$. Moreover, it follows from \eqref{eq1 proof lem DL17} (with $t=1-\lambda$) that
\begin{equation}\label{eq4 proof lem DL17}
	g(s+1-\lambda)\leq\dfrac{C_3}{1-\lambda} e^{\frac{-s}{a\lambda n}},
\end{equation}
for every $s>0$, where $C_3:=C_2(\mu(X))^{1/2n}$. Put
	$s_0=a\lambda n \log_+ (16C_1C_2C_3)$ so that
$$\frac{C_3}{1-\lambda} e^{\frac{-s_0}{a\lambda n}}<\frac{(1-\lambda)}{4C_1C_2}.$$
Then, by \eqref{eq4 proof lem DL17}, we have $g(s_0+1-\lambda)<\frac{(1-\lambda)}{4C_1C_2}$. 
Applying Lemma \ref{lem tg(t+s)} with $S=s_0+1-\lambda$ and $t_0=\frac{1-\lambda}{2}\in (0, 1-\lambda)$, we obtain
$g(s_0+2-2\lambda)=0$. Then
\begin{align*}
	u\geq\widehat{v}-s_0-2+2\lambda\geq \widehat{v}
	-C_4 a\lambda \left(1+\log_+\int_Xe^{\frac{2(V_\theta-u)}{a\lambda}}d\mu\right)-2,
\end{align*}
where $C_4>0$ is a constant depending only on $n, X, \omega_X, \theta, \varrho, B$ and $\beta$.
\end{proof}
\begin{lemma}\label{lem sup<cap}
		Let $u, v$ be negative $\theta$-psh functions such that
	\begin{itemize}
		\item $v\leq u+M$ for some $M\geq 1$;
		\item $v=\lambda v_1+(1-\lambda) v_2$, where  $v_1, v_2$ are $\theta$-psh functions, and $v_2$ has model singularity type and $\lambda\in (0, 1/2)$;
		\item there exist $C>0$ and $\alpha>0$ such that, for every $s>0$,
		\begin{equation}\label{eq0lemsup<cap}
			\int_{\{u<v-s\}}\theta_u^n\leq C[\capacity_{v_2}(u<v-s)]^{1+\alpha}.
		\end{equation}
		\end{itemize}
	Then, for every $\varepsilon>0$,
	\begin{equation*}
		\sup_{X }(v-u)\leq \varepsilon+\dfrac{2MC^{1/n}}{1-2^{-\alpha}}[\capacity_{v_2}(u<v-\varepsilon)]^{\alpha/n}.
	\end{equation*}
\end{lemma}
\begin{proof} {We adapt the same arguments as in~\cite{eyssidieux2009singular}.}
	By Lemma \ref{lem t^n cap}, for every $s>0$ and $0\leq t\leq 1-\lambda$, we have
		$$t^n \capacity_{v_2}\{u<v-s-t\}\leq \int_{\{u<v-s\}}\theta_u^n.$$
		Therefore, by \eqref{eq0lemsup<cap}, we obtain
		$$t^n \capacity_{v_2}\{u<v-s-t\}\leq C[\capacity_{v_2}(u<v-s)]^{1+\alpha},$$
		for every $s>0$ and $0\leq t\leq 1-\lambda$. Set $g(s)\coloneqq\left[\capacity_{v_2}(u<v-s)\right]^{1/n}$. We have
		\begin{equation*}
			t g(t+s)\leq C^{1/n}  g(s)^{1+\alpha},
		\end{equation*}
	for every $s>0$ and $0\leq t\leq 1-\lambda$.	
	If there exists $0\leq t_0\leq 1$ such that $g(\varepsilon)^{\alpha}=\frac{t_0}{2 C^{1/n}}$, then
	it follows from Lemma \ref{lem tg(t+s)} that $g\left(\varepsilon+\frac{t_0}{1-2^{-\alpha}}\right)=0$. Hence
	$$\sup_X (v-u)\leq \varepsilon+\dfrac{t_0}{1-2^{-\alpha}}\leq\varepsilon+\dfrac{2 C^{1/n}g(\varepsilon)^{\alpha}}{1-2^{-\alpha}},$$
    since $1\leq M$.
	Otherwise $2C^{1/n}g(\varepsilon)^{\alpha}\geq 1$, we infer that
		$$\sup_X(v-u)\leq M 
		\leq\varepsilon+\dfrac{2 M C^{1/n}}{1-2^{-\alpha}}g(\varepsilon)^{\alpha}.$$
This completes the proof.
\end{proof}
\begin{lemma}\label{lem extend GZ12 prop 5.3} Under the assumption of Lemma \ref{lem sup<cap}, assume further that $v_2\geq u$. We have
the following
	inequality 
$$\sup_{X } (v_1-u)\leq \left(2+\dfrac{2 M C^{1/n}}{\lambda^{1+\alpha}(1-2^{-\alpha})}\right)\|(v_1-u)_+\|_{L^1(\nu)}^{\frac{\alpha}{(n+1)\alpha+n}},$$
where $\nu=\mathbf{1}_{\{u<v_1\}}\theta_u^n$ and $(v_1-u)_+=\max(v_1-u,0)$.
\end{lemma}
\begin{proof}
	It follows from Lemma \ref{lem sup<cap} that for every $\varepsilon>0$, we have
	\begin{equation}\label{eq: lem35.1}
	\sup_X(v-u)\leq 2\varepsilon\lambda+\dfrac{2 M C^{1/n}}{1-2^{-\alpha}}[\capacity_{v_2}(u<v-2\varepsilon\lambda)]^{\alpha/n}.
    \end{equation}
	By Lemma \ref{lem t^n cap}, we obtain
		\begin{equation}\label{eq: lem35.2}
		    \begin{split}
		        (\varepsilon\lambda)^n\capacity_{v_2}(u<v-2\varepsilon\lambda)\leq \int_{\{u<v-\varepsilon\lambda\}}\theta_u^n\leq \int_{\{u<v_1-\varepsilon\}}\theta_u^n
		    \end{split}
		\end{equation} because $u\leq v_2$ so $\{u<v-\varepsilon\lambda\}\subset\{u<v_1-\varepsilon\}$. 
   Plugging \eqref{eq: lem35.2} into \eqref{eq: lem35.1}, since $v_2\geq u$ we obtain     
        \begin{align*}
		    \sup_X(v_1-u)&\leq 
2\varepsilon+\dfrac{2 M C^{1/n}}{\varepsilon^{\alpha}\lambda^{1+\alpha} (1-2^{-\alpha})}\left(\int_{\{u<v_1-\varepsilon\}}\theta_u^n\right)^{\alpha/n}
\\
&\leq 2\varepsilon+\dfrac{2 M C^{1/n}}{\varepsilon^{\alpha}\lambda^{1+\alpha} (1-2^{-\alpha}) }\left(\frac{1}{\varepsilon}\int_{\{u<v_1-\varepsilon\}}(v_1-u)_+d\nu\right)^{\alpha/n} \\
&\leq  2\varepsilon+\dfrac{2 M C^{1/n}}{\varepsilon^{\alpha(n+1)/n}\lambda^{1+\alpha}(1-2^{-\alpha}) }\|(v_1-u)_+\|_{L^1(\nu)}^{\alpha/n},
		\end{align*} 
        where the second inequality is due to Chebyshev’s inequality.
Choosing $\varepsilon:=\|(v_1-u)_+\|_{L^1(\nu)}^{\frac{\alpha}{(n+1)\alpha+n}}$ yields the desired inequality.
\end{proof}
\begin{proof}[Proof of Theorem~\ref{main}] 
	In this proof, we denote by	$M_j(a)$, for $j\geq 1$, 
    constants in $[1, \infty)$ only depending on $n, X, \omega_X, \theta, B, \beta$ and an upper bound for $\int_X e^{\frac{2(V_\theta-u)}{a\varepsilon_0}} d\mu$. We assume that $a\mapsto M_j(a)$ is decreasing for every
	$j\geq 1$, and $\lim_{a\to 0^+}M_{j}(a)=\infty$.

\smallskip
Since the cohomology class $\{\theta\}$ is big, there exists a negative 
$\theta$-psh function $\Psi_0$ with analytic
singularities such that $\theta +dd^c\Psi_0\geq\varepsilon_0\omega_X$
for some $\varepsilon_0>0$. 
It follows from~\cite{boucksom2004divisorial} that we can choose $\Psi_0$ so that ${\rm Amp}(\theta)=X\setminus {\rm Sing}(\Psi_0)$. By subtracting a positive constant, we can assume $\Psi_0\leq V_\theta$.

Fix $0<c_1\leq \frac{\varepsilon_0}{4A}$ 
and $0<\delta_1<\min\{\delta_0, \sqrt{\frac{\varepsilon_0}{4K}}\}$, where $A$, $K$, and $\delta_0$ are defined as in Lemma~\ref{lem kiselman}. For $0<\delta<\delta_1$,
we define
\begin{equation}\label{eq: constantB0}
B_0=B_0(c, \delta)=\frac{Ac+K^2\delta}{\varepsilon_0},
\end{equation}
with $c=c(\delta)$ being chosen hereafter. We set
\begin{equation}\label{eq: u c delta}
u_{c, \delta}:=B_0\Psi_0+(1-B_0)\Phi_{c, \delta},
\end{equation}
where $\Phi_{c, \delta}$ is defined by \eqref{kisleg}. It follows from Lemma~\ref{lem kiselman} that $$\theta+\ddc u_{c,\delta} \geq \varepsilon_0B_0^2\omega_X,$$ hence $u_{c,\delta}$ is $\theta$-psh. 
Recall that $a_0$ is a positive constant such that $\omega_X+a_0\ddc\psi\geq 0$. Fixing $0<a\leq a_0$, we set $$v_{a, \delta}:=u_{c, \delta}+ a\varepsilon_0B_0^2\psi.$$
\smallskip
 \textbf{Step 1: Bounding $\sup_X(v_{a,\delta}-u)$}. 
  We observe that $v_{a, \delta}$ is $\theta$-psh and $v_{a, \delta}\leq V_{\theta}+a \varepsilon_0B_0^2\psi.$ 
 Applying Lemma~\ref{lem DL17} with $\lambda=B_0$ and $v=v_{a, \delta}$, we obtain 
 \begin{equation}\label{eq1main}
 u\geq B_0v_{a, \delta}+(1-B_0)V_\theta -M_1(aB_0^2)\geq v_{a, \delta}-M_1(aB_0^2),
 \end{equation} using that $v_{a,\delta}\leq V_\theta$. 
 Since $v_{a, \delta}\leq V_{\theta}+aB_0^2 \varepsilon_0\psi,$ we also have $-\psi\leq \frac{V_{\theta}-u}{aB_0^2\varepsilon_0}$
 on $\{u<v_{a, \delta}\}$. We infer that for $s\geq 0$,
 \begin{align*}
\int_{\{u<v_{a, \delta}-s\}}\theta_{u}^n=\int_{\{u<v_{a, \delta}-s\}}e^{-\psi}d\mu 
		&\leq \int_{\{u<v_{a, \delta}-s\}}e^{\frac{V_{\theta}-u}{aB_0\varepsilon_0}}d\mu\\
&\leq \left(\int_{X}e^{\frac{2(V_{\theta}-u)}{aB_0^2\varepsilon_0}}d\mu\right)^{1/2} \left(\int_{\{u<v_{a, \delta}-s\}}d\mu\right)^{1/2}\\
&\leq M_2(aB_0^2)\left[\capacity_{V_{\theta}}(\{u<v_{a, \delta}-s\})\right]^{2},
 \end{align*}
 where the second inequality follows from H\"older's inequality and the last one follows from Theorem \ref{the exp cap}.
  Combining this with Lemma \ref{lem extend GZ12 prop 5.3} ($\lambda=B_0$, $v_1=v_{a,\delta}$, $v_2=V_\theta$), we infer that
\begin{equation}\label{eq2main}
\begin{split}
    \sup_X (v_{a, \delta}-u)&\leq \frac{M_3(aB_0^2)}{B_0^2}\|v_{a, \delta}-u\|^{\frac{1}{2n+1}}_{L^1(\nu)}\\
&\leq \frac{M_4(aB_0^2)}{B_0^2}\|(v_{a, \delta}-u)_+\|^{\frac{1}{2n+1}}_{L^1(\mu)},
\end{split}
\end{equation}
for every $s\geq 0$, where $\nu=\mathbf{1}_{\{u<v_{a, \delta}\}}\theta_u^n$.
Applying Lemma~\ref{lem DL17} with $\lambda=\frac{1}{2}$ and $v=\Psi_0+\varepsilon_0a\psi$, we obtain 
\begin{equation}\label{eq0.1main}
	u\geq \dfrac{\Psi_0+\varepsilon_0a\psi+V_{\theta}}{2}-M_5(a)\geq \Psi_0+\dfrac{\varepsilon_0a\psi}{2}-M_5(a).
\end{equation}
It follows that
\begin{align*}
v_{a, \delta}-u&=2B_0(\Psi_0+\dfrac{\varepsilon_0a\psi}{2}-u)+(1-B_0)(\Phi_{c, \delta}-u)\\
&\leq 2B_0 M_5(a)+(\Phi_{c, \delta}-u)_+.
\end{align*}
This combined with \eqref{eq2main} implies that

\begin{equation}\label{bound v a delta}
    \begin{split}
         \sup_X (v_{a, \delta}-u)&\leq \sup_X (v_{a, \delta}-u-2B_0 M_5(a))+2B_0 M_5(a)\\
    &\leq\frac{M_4(aB_0^2)}{B_0^2}\|(v_{a, \delta}-u-2B_0 M_5(a))_+\|^{\frac{1}{2n+1}}_{L^1(\mu)}+2B_0 M_5(a) \\
    &\leq\frac{M_4(aB_0^2)}{B_0^2}\|(\Phi_{c, \delta}-u)_+\|^{\frac{1}{2n+1}}_{L^1(\mu)}+2B_0 M_5(a) \\
      &\leq\frac{M_6(aB_0^2)}{B_0^2}\|(\Phi_{c, \delta}-u)_+\|^{\frac{\beta}{2n+1}}_{L^1(\omega_X^n)}+2B_0 M_5(a) \\
   &\leq\frac{M_6(aB_0^2)}{B_0^2}\|(\rho_{\delta}u-u)_+\|^{\frac{\beta}{2n+1}}_{L^1(\omega_X^n)}+2B_0 M_5(a)\\
    &\leq\frac{M_7(aB_0^2)}{B_0^2}\delta^{\frac{2\beta}{2n+1}}+2B_0 M_5(a),
    \end{split}
\end{equation}
where the last estimate holds due to Lemma \ref{Jen}.
Fix $U\Subset \textrm{Amp}(\theta)\setminus\{
\psi=-\infty\}$. 

\smallskip
\noindent\textbf{Step 2: Bounding $\sup_U(u_{c,\delta}-u)$.}
Recall that $u_{c,\delta}=v_{a,\delta}-aB_0^2\varepsilon_0\psi$. Set $m_U:=\sup_{\Bar{U}}(-\psi)$.
It follows from \eqref{bound v a delta} that
$$\sup_U(u_{c, \delta}-u)\leq aB_0^2\varepsilon_0m_U+\frac{M_7(aB_0^2)}{B_0}\delta^{\frac{2\beta}{2n+1}}+2B_0 M_5(a).$$
Denoting by $h_1$ the function $a\mapsto a\varepsilon_0m_U+2M_5(a)$ and choosing $$a\coloneqq\min\left\{h_1^{-1}\left(\frac{1}{\sqrt{B_0}}\right),\frac{a_0}{2}\right\},$$ since $B_0^2\leq B_0$ we have
$$\sup_U(u_{c, \delta}-u)\leq \sqrt{B_0}+\frac{M_7\left(h_1^{-1}\left(\frac{1}{\sqrt{B_0}}\right)B_0^2\right)}{B_0^2}\delta^{\frac{2\beta}{2n+1}}.$$
Denoting by $h_2$ the function $t\mapsto \frac{M_7\left(h_1^{-1}\left(\frac{1}{\sqrt{t}}\right)t^2\right)}{\sqrt{t^5}}$ and choosing $c=c(\delta)>0$ so that
$$h_2(B_0)=h\left( \frac{Ac+K\delta^2}{\varepsilon_0}\right)=\delta^{\frac{-2\beta}{2n+1}}.$$ 
Thus, we obtain
$$\sup_U(u_{c, \delta}-u)\leq 2\sqrt{B_0(\delta)} .$$
\noindent \textbf{Step 3: Conclusion.}
As follows from Lemma \ref{lem thmD Dem et al} below, we obtain
\begin{equation*}
    \rho_{\kappa(\delta)}u(z)-u(z)\leq 4\sqrt{B_0(\delta)}+2C_U B_0(\delta)+2K^2\delta,
\end{equation*} where $$\kappa(\delta)=\delta\exp\Big(\dfrac{-4A(C_UB_0+2\sqrt{B_0(\delta)}+2K^2\delta)}{\varepsilon_0B_0} \Big),\quad C_U=\sup_U(V_\theta-\Psi_0).$$
Replacing $\delta$ with $\kappa (\delta)$ we obtain for $\delta\leq \kappa(\delta_0)$,
\[\rho_\delta u(z)-u(z)\leq 4\sqrt{B_0(\kappa^{-1} (\delta))}+2C_U B_0(\kappa^{-1} (\delta)). \]
The conclusion thus follows.
\end{proof}

\begin{lemma}\label{lem thmD Dem et al}
	Let $c: (0, \delta_1)\rightarrow (0, c_1)$ and $B_1:\R^+\rightarrow\R^+$ be functions. Assume $\lim_{\delta\to 0^+}c(\delta)/\delta=+\infty$, so that $K^2\delta<Ac(\delta)$ for every $0<\delta<\delta_2$, where $\delta_2\in (0, \delta_1)$ is sufficiently small. 
    Let $u\in\PSH(X,\theta)$ be such that
	\begin{equation}\label{eq0 lem thmD Dem et al}
	\sup_{U}(u_{c(\delta), \delta}-u)\leq  B_1(\delta),
	\end{equation}
	for every $\delta\in(0,\delta_2)$, where $u_{c(\delta),\delta}$ is defined by \eqref{eq: u c delta} and 
    $U\Subset \textrm{Amp}(\theta)\backslash\{\psi=-\infty\}$. Then we have \begin{equation}\label{eq: delta_diff}
	\sup_U(\rho_{\kappa(\delta)}u-u)\leq 2(B_1+C_UB_0+2K^2\delta),
	\end{equation}
where
	\begin{equation*}
	    C_U=\sup_U(V_{\theta}-\Psi_0) \;\text{and}\; \kappa(\delta)=\delta\exp\Big(\dfrac{-4A(C_UB_0+B_1+2K^2\delta)}{\varepsilon_0B_0} \Big).
	\end{equation*}
\end{lemma}
\begin{proof}The argument follows from \cite[pages 642-643]{demailly2014holder}, we include the proof
for the readers’ convenience.
	By the assumption, we have
\begin{align*}
B_1\geq u_{c, \delta}(z)-u(z)&=B_0(\Psi_0(z)-u(z))\\
&+
(1-B_0)\left(\rho_{t_0}u(z)+K(t_0^2-\delta^2)+K(t_0-\delta)-c\log\left(\frac{t_0}{\delta}\right)-u(z)\right),	
\end{align*}
for every $z\in U$, where $t_0=t_0(z)\in (0, \delta)$ is a minimum point of 
$\rho_{t}u(z)+Kt^2+Kt-c\log\left(\frac{t}{\delta}\right)$. Since $V_{\theta}\geq u$
and $\rho_{t_0}u+Kt_0^2\geq u$ (see Lemma \ref{lem kiselman}), it follows that
$$B_0(\Psi_0(z)-V_{\theta}(z))
-c(1-B_0)\log\left(\frac{t_0}{\delta}\right)\leq  B_1+K\delta^2+K\delta\leq B_1+2K^2\delta,$$
for every $z\in U$. Hence, for every $z\in U$,
\begin{align*}
t_0(z)&\geq \delta\exp\bigg(\dfrac{B_0(\Psi_0(z)-V_{\theta}(z))-B_1-2K^2\delta}{c(1-B_0)} \bigg)\\
&\geq \delta\exp\Big(\dfrac{-4C_UB_0-4B_1-8K^2\delta}{A^{-1}\varepsilon_0B_0} \Big)
= \kappa(\delta),
\end{align*}
 using the facts that $B_0\varepsilon_0=Ac+K^2\delta\leq 2Ac$ and $1-B_0\geq \frac{1}{2}$. Since $t\mapsto \rho_tu+Kt^2$ is increasing, we have
\begin{align*}
    \rho_{\kappa(\delta)}u(z)-u(z)
	&\leq \rho_{t_0}u(z)+Kt_0^2-u(z)\\
		&\leq \left(\rho_{t_0}u(z)+K(t_0^2-\delta^2)+K(t_0-\delta)-c\log\left(\frac{t_0}{\delta}\right)\right)-u(z)+2K^2\delta\\
	&\leq \Phi_{c,\delta}(z)-u(z)+2K^2\delta\\
	&=\dfrac{1}{1-B_0}(u_{c, \delta}(z)-u(z))+\dfrac{B_0}{1-B_0}(u(z)-\Psi_0(z))+2K^2\delta\\
	&\leq \dfrac{B_1(\delta)}{1-B_0}+\dfrac{B_0}{1-B_0}(V_{\theta}(z)-\Psi_0(z))+2K^2\delta,
\end{align*}	
for every $z\in U$, where the last inequality holds due to \eqref{eq0 lem thmD Dem et al}. Therefore,
\begin{center}
	$\rho_{\kappa(\delta)}u(z)-u(z)\leq \dfrac{B_1}{1-B_0}+\dfrac{C_UB_0}{1-B_0}+2K^2\delta\leq 2(B_1+B_0C_U)+2K^2\delta,$
\end{center}
for every $z\in U$.
The proof is thus complete.
\end{proof}

We now deal with the special case when $\theta =\omega_X$ is a K\"ahler form, proving Corollary~\ref{coro: main}. 
Recall that $\mathcal{N}=\mathcal{N}(B,\beta,a_0,C_0,h)$ is the set of probability measures $\nu$ on $X$ such that $\nu=e^{-\psi}\mu$ with $\mu\in\mathcal{M}(B,\beta)$, $\psi\in\PSH(X, a_0\omega_X)$ and  
$$\int_Xh(-\psi)e^{-\psi}{\rm d}\mu\leq C_0,$$ 
where $a_0>0$, $ C_0>0$ are fixed constants and $h: (0, \infty)\rightarrow (0, \infty)$ is an increasing concave function with $h(\infty)=\infty$. We define the set $$\mathcal{F}=\{u\in\mathcal{E}(X,\omega_X),\,\sup_X u=0\colon (\omega_X+\ddc u)^n\in\mathcal{N}\}.$$
\begin{lemma}\label{lem: compact}
   The closure of $\mathcal{F}$ is contained in $\mathcal{E}(X,\omega_X)$. 
\end{lemma}
\begin{proof} Let $(u_j)_{j\in\mathbb N}\subset \mathcal{F}$ be a sequence of $\omega_X$-psh functions such that $(\omega_X+\ddc u_j)^n\in\mathcal{N}$ and $\sup_X u_j=0$. By Hartogs' lemma (see, e.g., \cite[Proposition 8.4]{guedj2017degenerate}), there exists a subsequence, still denoted $(u_j)$, such that
$u_j\xrightarrow{L^1} u$ for some $u\in\PSH(X, \omega_X)$. 
    By Proposition~\ref{prop: unif-capa}, the family of non-pluripolar measures $(\omega_X+\ddc u_j)^n$ is uniformly absolutely continuous with respect to capacity. Set $v_j=\max(u_j,u)$. We see that $v_j$ converges to $u$ in capacity thanks to Hartogs' lemma and the quasi-continuity of $\omega_X$-psh functions.
    As follows from~\cite[Propositions 2.10]{dinew12-cv-capacity}, the sequence of measures $(\omega_X+\ddc v_j)^n$ is uniformly absolutely continuous with respect to capacity. Applying~\cite[Proposition 2.11]{dinew12-cv-capacity}, we conclude that $u\in\mathcal{E}(X,\omega_X)$.
    Therefore, the closure of $\mathcal{F}$ in the $L^1$ topology is contained in $\mathcal{E}(X,\omega_X)$.
\end{proof}
\begin{proof}[Proof of Corollary~\ref{coro: main}] 
    By Theorem~\ref{main}, it suffices to show that there exists a positive constant $C=C(a_0,X,\omega_X,\mu)$ such that
    \begin{equation}\label{eq: skoda}
        \int_X e^{-2u/a_0}d\mu\leq C.
    \end{equation}
    Thanks to Lemma~\ref{lem: compact}, the closure of the set $\mathcal{F}= \{u\in\mathcal{E}(X,\omega_X),\,\sup_X u=0\colon (\omega_X+\ddc u)^n\in\mathcal{N}\}$ is a compact family of functions in $\mathcal{E}(X,\omega_X)$. Since functions in $\mathcal{E}(X,\omega_X)$ have zero Lelong number at every point on $X$ by \cite{guedj2007weighted}, the uniform Skoda integrability
theorem (see, e.g.,~\cite{guedj2017degenerate}) ensures that there exists a constant $C$ depending on $a_0$ and $\overline{\mathcal{F}}$ such that the inequality \eqref{eq: skoda} holds for all $u\in\mathcal{F}$. 
\end{proof}
\begin{remark}
     Lemma~\ref{lem: compact} still holds in the more general setting where $\theta$ is merely semi-positive and big (i.e., $\int_X \theta^n > 0$). Corollary \ref{coro: main} also remains valid, but in this case only on compact subsets of $\textrm{Amp}(\theta) \setminus \{\psi = -\infty\}$.
\end{remark}
\bibliographystyle{alpha}
	\bibliography{bibfile}	

\end{document}